\documentclass[12pt]{article}

\usepackage{amssymb}
\usepackage{amsmath}
\usepackage{amsbsy}
\usepackage{amscd}
\usepackage{amsfonts}
\usepackage{amsthm}
\usepackage{mathrsfs}
\usepackage{verbatim}
\usepackage{hyperref}
\usepackage{fullpage}
\usepackage{mathdots}
\usepackage{graphicx,subfigure}
\usepackage[english]{babel}
\usepackage[utf8]{inputenc}
\usepackage{tikz}
\usepackage{adjustbox}
\usepackage{authblk}
\usepackage{caption}

\usepackage{mathtext}
\usepackage{stackengine,scalerel}

\usepackage{tikz-cd}
\usetikzlibrary{backgrounds,fit, matrix}
\usetikzlibrary{positioning}
\usetikzlibrary{calc,through,chains}
\usetikzlibrary{arrows,shapes,snakes,automata, petri}
\usepackage[boxsize=1.25em, centerboxes]{ytableau}

\newtheorem{Theorem}[equation]{Theorem}
\newtheorem{Corollary}[equation]{Corollary}
\newtheorem{Lemma}[equation]{Lemma}
\newtheorem{Proposition}[equation]{Proposition}

\theoremstyle{definition}
\newtheorem{Definition}[equation]{Definition}
\newtheorem{Example}[equation]{Example}

\newtheorem{Remark}[equation]{Remark}

\numberwithin{equation}{section}
\numberwithin{figure}{section}

\newcommand{\PP}{{\mathbb P}}
\newcommand{\C}{{\mathbb C}}
\newcommand{\Z}{{\mathbb Z}}

\newcommand{\R}{{\mathbb R}}

\newcommand{\mc}[1]{\mathcal{#1}}

\usepackage[all,2cell]{xy}
	\usepackage{amssymb, hyperref}
	\usepackage[all,2cell]{xy}
	\usepackage{hyperref}
	\usepackage{array}
	\usepackage[capitalise]{cleveref}

\begin{document}
	
\title{From Schubert Varieties to Doubly-Spherical Varieties}

\author[1]{Mahir Bilen Can}
\author[2]{S. Senthamarai Kannan}
\author[3]{Pinakinath Saha}

\affil[1]{\small{Tulane University, New Orleans, USA, mahirbilencan@gmail.com}}
\affil[2]{\small{Chennai Mathematical Insitute, Chennai, India, kannan@cmi.ac.in}}
\affil[3]{\small{Indian Institute of Science, Bangaluru, India, pinakinaths@iisc.ac.in}}

\maketitle

\begin{abstract}
Horospherical Schubert varieties are determined. 
It is shown that the stabilizer of an arbitrary point in a Schubert variety is a strongly solvable algebraic group.
The connectedness of this stabilizer subgroup is discussed.
Moreover, a new family of spherical varieties, called doubly spherical varieties, is introduced.
It is shown that every nearly toric Schubert variety is doubly spherical.
\vspace{.5cm}

\medskip
		

\end{abstract}

\section{Introduction}

In this article, we discuss the interplay between Schubert varieties and spherical varieties. 
Our primary focus is on understanding the orbits of maximal reductive subgroups in Schubert varieties by leveraging tools from spherical geometry. 
Towards this end, we examine the stabilizer subgroups of points in general positions within Schubert varieties.
This analysis led us to consider a new family of spherical varieties and provide characterizations for several other significant families of spherical homogeneous varieties within the context of Schubert varieties. These include horospherical Schubert varieties and strongly solvable spherical varieties, both of which play pivotal roles in the development of algebraic group actions. 
We proceed with a brief motivation for investigating horospherical varieties, followed by a detailed presentation of our findings concerning them in the context of Schubert varieties.

Let $G$ denote a connected reductive algebraic group defined over an algebraically closed field $k$ of arbitrary characteristic unless stated otherwise. 
A closed subgroup $H\subseteq G$ is called a {\em horospherical subgroup} if it contains a maximal unipotent subgroup $U$ of $G$. 
We fix a Borel subgroup $B\subset G$ and a maximal torus $T\subset B$ such that $B=TU$.
If $H$ contains $U$, then $B^- H$ is open in $G$, where $B^-$ stands for the unique Borel subgroup opposite to $B$.
This means that the homogeneous space $G/H$ is a {\em horospherical homogeneous space}. 
More generally, a normal $G$-variety $X$ is said to be {\em horospherical} if the stabilizer of a point in general position $x\in X$ is a horospherical subgroup of $G$.
In this case, any Borel subgroup of $G$ has an open orbit in $X$, implying that $X$ is a {\em spherical $G$-variety}. 
Horospherical varieties play a crucial role not only in understanding spherical varieties but also in studying general actions of $G$.
For instance, a key result by Knop~\cite[Satz 2.7]{Knop1990} asserts that for any $G$-variety $X$, there exists a one-parameter family of nonsingular $G$-varieties $X_t$ ($t$ belonging to the affine line). 
For non-zero values of $t$, $X_t$ is $G$-equivariantly isomorphic to an open, nonsingular, $G$-stable subset of $X$. 
Moreover, the fiber at $t=0$ is a product of the form $V\times X_0$, where $X_0$ is a nonsingular horospherical $G$-variety, and $G$-action on $V$ is trivial. 
It is also worth noting that horospherical varieties have applications in representation theory. 
In fact, an important characterizing property of quasi-affine horospherical varieties is the grading on their coordinate rings, induced by the monoid of dominant weights which satisfies the property
$$
\C[X]_{(\lambda)} \C[X]_{(\mu)} \subseteq \C[X]_{(\lambda+\mu)}.
$$
Here, $\C[X]_{(\lambda)}$ (resp. $\C[X]_{(\mu)}$) is the isotypic component of type $\lambda$ (resp. of $\mu$) of $\C[X]$. 
\medskip

Let $\mc{R}$ denote the root system of the pair $(G,T)$, and let $\mc{S}$ denote its subset consisting of the simple roots determined by $(G,B,T)$.
Let $W$ denote the Weyl group of $G$ with respect to $T$. 
For $w\in W$, let $\dot{w}$ denote a representative of $w$ in $N_{G}(T)$, the normalizer of $T$ in $G$.
For $w\in W,$ let $X_{wB}$ denote the corresponding Schubert variety, $X_{wB}:=\overline{B\dot{w}B/B} \subseteq G/B$.
Since the left $B$-action on $G/B$ stabilizes $X_{wB}$, the stabilizer group of $X_{wB}$ in $G$ is given by the parabolic subgroup
determined by the set of simple roots $\mc{I}_w := \{ \alpha \in \mc{S} \mid w^{-1}(\alpha) < 0 \}$. 
For $J\subseteq \mc{I}_{w}$, $P_{J}$ denotes the parabolic subgroup of $G$ corresponding to $J$, and $L_J$ denotes its Levi subgroup containing $T$.
This gives rise to a finite family of reductive groups parametrized by the subsets of $\mc{I}_{w}$ each of which acts on $X_{wB}$.

The starting point of our paper is the following question:
\begin{quote}
For which $w \in W$ is the Schubert variety $X_{wB}$ a horospherical $L_{J}$-variety?
\end{quote}
Recently, it was shown in~\cite{CanSaha2023, GaoHodgesYong} that $X_{wB}$ is a spherical $L_{J}$-variety if and only if $w$ can be expressed as $w = w_{0,J}c$, where $\ell(w) = \ell(w_{0,J}) + \ell(c)$. 
Here, $w_{0,J}$ denotes the longest element of the Weyl group of $L_{J}$, and $c$ is a {\em Coxeter-type element} of $W$, meaning a product of distinct simple reflections from a subset of $S$ in a specified order. 
The set $S$ corresponds to the simple reflections determined by $\mathcal{S}$.
Leveraging this characterization of spherical Schubert varieties, we now turn to the combinatorial criteria for identifying when a Schubert variety is horospherical. 

\begin{Theorem}\label{intro:T1}
The Schubert variety $X_{wB} \subseteq G/B$ is a horospherical $L_{J}$-variety if and only if $w$ can be written in the form $w= w_{0,J}c$, where $\ell(w) = \ell(w_{0,J}) + \ell(c)$, and the supports of $w_{0,J}$ and $c$ are disjoint subsets of $S$. 
(The support of an element $v\in W$ is the set of simple reflections from $S$ that appear in a reduced expression of $v$.)
\end{Theorem}

After Theorem~\ref{intro:T1}, it becomes a natural task to investigate the isotropy groups for the action of $L_J$ on Schubert varieties. 
To this end, we examined the relationship between Schubert varieties and the subgroups of the Borel subgroup of the Levi subgroups.

A subgroup $H \subseteq G$ is said to be {\em strongly solvable} if it is contained in a Borel subgroup of $G$. 
These strongly solvable groups are of intrinsic interest and have significant applications in the classification of spherical subgroups. 
This was first recognized by Luna in his work~\cite{Luna1993}, where he classified strongly solvable spherical subgroups. 
Subsequently, Avdeev provided an alternative and more detailed approach to classifying these subgroups in two separate publications~\cite{Avdeev2011,Avdeev2015}.

We establish a connection between the theory of strongly solvable spherical subgroups and our theory of spherical Schubert varieties. 
In fact, we show that our result follows from a more general statement: the stabilizer of any point $x \in X_{wB}$ in $L_{J}$ is strongly solvable. 
We state this as a second main result of our paper.

\begin{Theorem}\label{intro:T2}
Let $w\in W$ and $J\subseteq \mc{I}_w$.
Let $X_{wB}$ denote the corresponding Schubert variety. 
Then the stabilizer in $L_{J}$ of every point in $X_{wB}$ is a strongly solvable subgroup of $L_{J}$.
\end{Theorem}

This theorem easily adapts to the spherical Schubert varieties.

In the rest of the paper, by ``a point in general position'' we mean the following. Let $X$ be an irreducible $K$-variety, where $K$ is an algebraic group. A point $x\in X$ is said to be a {\em point in general position} if the dimension of the $K$-orbit of $x$ is maximal.

In particular, as a corollary we show in Proposition~\ref{P:ssss} that the stabilizer of a point in general position of a {spherical} Schubert variety is a strongly solvable spherical subgroup of the appropriate Levi subgroup.

We now proceed to discuss an entirely new family of spherical varieties.
\begin{Definition}
Let $X$ be a spherical $G$-variety. 
If every $B$-orbit closure in $X$ is a spherical $L$-variety for some Levi subgroup $L\subseteq G$, then we call $G$ a {\em doubly-spherical $G$-variety}.
\end{Definition}

There are various intriguing representation theoretic properties of doubly-spherical varieties. 
We mention a characterization of the affine doubly-spherical $G$-varieties. 

\begin{Proposition}
Let $X$ be a spherical affine $G$-variety. 
Let $A$ denote the coordinate ring of $X$. 
Then $X$ is a doubly-spherical affine $G$-variety if and only if every $B$-stable prime ideal of $A$ is a multiplicity-free $L$-module for some Levi subgroup $L\subseteq G$. 
\end{Proposition}

This result is a rather direct consequence of a theorem of Kimel'feld and Vinberg~\cite{KimelfeldVinberg} which characterizes affine spherical $G$-varieties according to their coordinate ring being a multiplicity-free $G$-module or not. 
In this manuscript we are interested in answering the following question:
\begin{center}
Which Schubert varieties are doubly-spherical?
\end{center}
We answer this question for Schubert varieties that belongs to another interesting family of spherical varieties.

\begin{Definition}
Let $X$ be a spherical $G$-variety. 
Let $T$ be a maximal torus of $G$. 
If the minimum codimension of $T$-orbit in $X$ is 1, then we call $X$ a {\em nearly toric $G$-variety}. 
\end{Definition}

Nearly toric $G$-varieties were originally considered in~\cite{CanDiaz2024}, 
there the authors classified the nearly toric Schubert varieties of type A. 
In the present article, we obtain two theorems on the nearly toric Schubert varieties.
The first of these two results gives a combinatorial characterization.

\begin{Theorem}\label{intro:T3}
Let $w\in W$. 
Then $X_{wB}$ is a nearly toric Schubert variety if and only if 
$w$ is of the form $w= s_\alpha c$, where $s_\alpha$ is a simple reflection, and $c$ is a Coxeter type element such that $s_\alpha \in {\rm supp}(c)$ and $\ell(w) = \ell(c)+1$. 
\end{Theorem}

Our second main result shows that nearly spherical Schubert varieties are special instances of the doubly-spherical Schubert varieties.

\begin{Theorem}\label{intro:T4}
Let $X_{wB}$ be a nearly toric $L_{\mc{I}_w}$-variety. 
Then $X_{wB}$ is a doubly-spherical $L_{\mc{I}_w}$-variety. 
\end{Theorem}

Next, we will discuss the structure of our paper and mention some additional results. 
In the next preliminaries section, we setup our notation and review some basic notions from algebraic group theory and setup our notation.
In Section~\ref{S:Horospherical} we prove our first main result, Theorem~\ref{intro:T1}, regarding the horospherical Schubert varieties.
Once we obtain the characterization of the horospherical Schubert varieties, it becomes a natural question to identify which of these Shubert varieties are nonsingular, or have what sort of combinatorial properties. 
We discuss some of these questions in the same section.
The purpose of Section~\ref{S:Existence} is to address the question of finding a nonsingular horospherical Schubert variety whose stabilizer in $G$ is a prescribed parabolic subgroup. We answer this question affirmatively in Theorem~\ref{T:prescribed}.
In Section~\ref{S:Generic}, we prove our second main result, that is, Theorem~\ref{intro:T2}. 
In Section~\ref{S:Closed}, we discuss the closed $L_{J}$-orbits in the Schubert variety $X_{wB}$. 
In Theorem~\ref{T:itissimple}, we show that the Schubert variety $X_{wB}$ has a unique closed $L_{J}$-orbit if and only if $w=w_{0,J}$.
This theorem yields at once that a Schubert variety is a wonderful variety if and only if its isomorphic to a full flag variety. 
(Wonderful varieties are introduced at the beginning of Section~\ref{S:Closed}.) The purpose of Section~\ref{S:Connectedness} is to discuss the connectedness properties of the stabilizers of points in general positions. 
It turns out that this is a delicate question. 
The stabilizer subgroups of points in general position turn out to be connected if we assume that $G$ is of adjoint type.
However, if we do not assume the adjointness of $G$, counterexamples are found. 
We close our paper by Section~\ref{S:Doubly} where we prove our last two Theorems~\ref{intro:T3} and~\ref{intro:T4}.

\section{Preliminaries}\label{S:Preliminaries}

We introduced some of our notation in the introductory section, but here we will establish additional notation and clarify certain terminology. Before proceeding, we would like to direct beginners to a few useful resources on Schubert varieties and spherical varieties. Throughout most of this article, our focus will be on the Schubert varieties of the full flag variety of a connected reductive group $G$. For a thorough introduction to Schubert varieties and their geometry, we recommend the references~\cite{Brion_Lectures, BrionKumar, BrownLakshmibai}. For spherical varieties, some good resources are~\cite{Timashev}, \cite{Perrin2018}, and~\cite{Gandini}.

Let $G$ be a connected reductive group. Since the central torus is contained in every Borel subgroup, the assumption that $G$ is a connected semisimple group does not weaken any of our results. Consequently, we will adopt this assumption throughout the paper. Additionally, where appropriate, we will assume that $G$ is simply connected; this assumption likewise does not cause loss of generality.

The set of positive (resp. negative) roots determined by the triplet $(G,B,T)$ is denoted by $\mc{R}^+$ (resp. by $\mc{R}^-$).
It is convenient to write $\beta >0$ to indicate that $\beta$ is an element of $\mc{R}^+$. 
Similarly, $\beta < 0$ means that $\beta \in \mc{R}^-$. 
The set of simple reflections of $W$ determined by $B$ is denoted by $S$.  
For $w\in W$, the support of $w$ is defined by 
$$
\mathrm{supp}(w) := \{s\in S\mid s \leqslant w\}.
$$
Here, $\leqslant$ stands for the Bruhat-Chevalley order that is defined by 
$$
v\leqslant w \iff X_{vB} \subseteq X_{wB} \qquad (v,w\in W).
$$
If the set of simple roots is given by $\mc{S}=\{\alpha_1,\,\ldots,\,\alpha_n\}$, then 
the simple reflection in $W$ corresponding to the simple root $\alpha_i$, where $1\leq i \leq n$, will be denoted by $s_{i}$. 
Hence, in the notation of the introduction, we have $S = \{s_1,\dots, s_n\}$.
For $w\in W$, the {\em length} of $w$, denoted by $\ell(w)$, is the minimum number of simple reflections required to write $w$ as a product.
Then we have the {\em length function}, $w\mapsto \ell(w)$, $w\in W$, which is identical to the dimension function $w\mapsto \dim X_{wB}$, $w\in W$.
Another formulation of the length function is as follows. 
For $w\in W$, we set 
$$
\mc{R}^{+}(w):=\{\beta\in \mc{R}^{+}: w(\beta)<0\}\quad\text{and}\quad \mc{R}^+(w^{-1}):=\{\beta\in \mc{R}^+ : w^{-1}(\beta)<0\}.
$$ 
Then we have $\ell(w) = | \mc{R}^+(w)| = | \mc{R}^+(w^{-1})|$. 
We now consider in passing the set of simple roots that are contained in $\mc{R}^+(w^{-1})$: 
$$
\mc{I}_w:= \mc{S}\cap \mc{R}^+(w^{-1}).
$$
It turns out that this subset of simple roots retains a great deal of information about the geometry of $X_{wB}$. 
Indeed, the stabilizer of the Schubert variety $X_{wB}$ in $G$ is given by the parabolic subgroup
generated by the Borel subgroup $B$ together with all root subgroups $U_{-\alpha}$, where $\alpha \in \mc{I}_w$. 
This is a consequence of the fact that the {\em left descent set} of $w$, that is $J(w):=\{ s \in S \mid \ell(s  w)<\ell(w)  \}$, is given by 
$$
J(w)=\{ s_\alpha \mid \alpha \in \mc{I}_w\}.
$$
Hereafter, for $\mc{J}\subseteq \mc{S}$, by $P_{\mc{J}}$ we will denote the corresponding {\em standard parabolic subgroup} generated by $B$ and all the root subgroups $U_{-\alpha}$, where $\alpha \in \mc{J}$. 
The {\em standard Levi subgroup} of $P_{\mc{J}}$ is the unique Levi subgroup containing $T$. 
Such a Levi subgroup will be denoted by $L_{\mc{J}}$.

\medskip

A product of simple reflections $s_{i_1}s_{i_2}\cdots s_{i_r}$ is called a {\em reduced expression} of $w$ if 
the following holds: 
$$
w= s_{i_1}s_{i_2}\cdots s_{i_r}\quad\text{ and }\quad\ell(w) = r.
$$ 
A {\em Coxeter type element} in $W$ is an element of the form $s_{i_1}s_{i_2}\cdots s_{i_r}$, where $s_{i_1},\dots, s_{i_r}$ are mutually distinct simple reflections. 
Reduced expressions are useful for understanding the Bruhat-Chevalley order. 
In particular, the theorem of ``Subword Property'' asserts that, for a reduced expression $w=s_{i_1}\cdots s_{i_k}$
in $W$ and $u\in W$, we have $u\leqslant w$ if and only if a reduced expression of $u$ is a subword of $s_{i_1}\cdots s_{i_k}$.
Here, by a {\em subword} we mean a product of the form $s_{i_{j_1}} s_{i_{j_2}} \cdots s_{i_{j_l}}$ for some 
$i_1 \leq i_{j_1} < i_{j_2} <\cdots < i_{j_l} \leq i_k$.

Let $\beta\in \mc{R}$.
The {\em root subgroup associated with $\beta$} is a one-dimensional unipotent group, denoted by $U_\beta$, and defined 
as the image of a homomorphism $x_\beta : (k,+)\to G$ satisfying the following identity 
for every $a\in k$ and $t\in T$: 
\begin{align*}
t x_\beta(a) t^{-1} = x_\beta(\beta(t)a).
\end{align*}
Then $G$ (resp. $B$) is generated by $T$ and all $U_{\pm \alpha}$ (resp. $T$ and all $U_\alpha$), where $\alpha \in \mc{S}$.
It also worth noting that for every root $\beta\in \mc{R}$, there is an element $w\in W$ and a simple root $\alpha\in \mc{S}$ such that $\beta = w\alpha$
implying that the corresponding 1-dimensional unipotent subgroup $U_{\beta}$ is given by $U_\beta = \dot{w} U_{\alpha} \dot{w}^{-1}$.

The one-dimensional multiplicative group will be denoted by $\mathbb{G}_m$.

\section{Horospherical Elements of $W$}\label{S:Horospherical}

Horospherical varieties have a pivotal place in the theory of spherical varieties as shown by Brion and Pauer in~\cite{BrionPauer}.  
In this section we will identify the horospherical Schubert varieties. 
It will not surprise the reader to find that the structure of the horospherical Schubert varieties has many interesting features. 
We begin with proving one direction of our first main theorem.

\begin{Proposition}\label{P:horo1}
Let $w\in W$ be an element such that $w = w_{0,J} c$ for some $J\subseteq \mc{S}$ and $\ell(w) = \ell(w_{0,J}) + \ell(c)$. 
If, in addition, the supports of $c$ and $w_{0, J}$ are disjoint, then the Schubert variety $X_{wB}$ is a horospherical $L_{J}$-variety. 
\end{Proposition}

\begin{proof}
We fix a reduced expression for $c$: 
$$
c=s_{i_{1}}s_{i_{2}}\cdots s_{i_{r}}.
$$ 
For $1\le j\le r$, set $\beta_{i_j} :=s_{i_1}\cdots s_{i_{j-1}}(\alpha_{i_{j}})$. 
Then the open cell of $X_{cB}$ is given by  
\[
U\dot{c}B/B=U_{\beta_{i_{1}}}U_{\beta_{i_{2}}}\cdots U_{\beta_{i_{r}}}\dot{c}B/B.
\]
Let $\xi \in U\dot{c}B/B$ be the point defined by $\xi :=u_{\beta_{i_{1}}}(1)\cdots u_{\beta_{i_{r}}}(1)\dot{c}B/B$. 
Let $H$ denote the stabilizer of $\xi$ in $L_{J}$.
We seek a generating set for $H$. 
Let $T_c$ denote the diagonalizable group  
$$
T_{c}:=\bigcap\limits_{j=1}^{r} \ker \beta_{i_{j}}.
$$
Since every element of $T_c$ fixes $\xi$, $T_c$ is contained in $H$. 
At the same time, since the support of $c$ and $w_{0,J}$ do not have any common elements, 
every root subgroup $U_\alpha$, where $\alpha \in J$ stabilizes $\xi$. 
Hence, we see that $\big \langle U_{\alpha} \mid \alpha \in J\big\rangle$ is a subgroup of $H$.
It follows that we have the inclusion,
$$
\big  \langle T_c, \ U_{\alpha} \mid \alpha \in J\big\rangle \subseteq H.
$$ 
Since $\dim T_c = n-r$ and the dimension of the unipotent group $\prod_{\alpha \in \mc{R}^{+}(w_{0,J}^{-1})} U_\alpha$ is given by $\ell(w_{0,J})$,
we have the inequality, 
$$
n-r +\ell(w_{0,J})\ \leq \ \dim H.
$$
Evidently, the dimension of $L_{J}$ is given by $\dim L_{J}=n+2\ell(w_{0,J})$.
Since $L_{J}$ has an open orbit in $X_{wB}$, and since $\xi$ is a point in $X_{wB}$, 
the following inequalities follow: 
\begin{align*}
\dim X_{wB} & \leq \dim L_{J} - \dim H \\
& \leq \ (n+2\ell(w_{0,J}))- (n-r+\ell(w_{0,J})) \\
& = r + \ell(w_{0,J}).
\end{align*}
But we already know that $\dim X_{wB} = \ell(w) = \ell (w_{0,J})+\ell(c)$, implying that 
$$
\dim X_{wB} = \dim L_{J} - \dim H.
$$
In other words, the dimension of the orbit $L_{J}\cdot \xi \cong L_{J}/H$ is maximal. 
By the uniqueness of the open orbit, this orbit must be the open orbit. 
Therefore, the stabilizer of a point in general position in $X_{wB}$ is conjugate to $H$.
Since $H$ contains the the subgroup $\langle U_{\alpha}: \alpha\in J\rangle$, we see that $L_{J}/H$ is a horospherical homogeneous space. 
This finishes the proof of our assertion. 
\end{proof}

We now proceed to prove the converse of our previous proposition.

\begin{Proposition}\label{P:horo2}
Let $w\in W$. 
Let $X_{wB}$ denote the corresponding Schubert variety. Let $J\subseteq \mc{I}_{w}$.
If $X_{wB}$ is a horospherical $L_{J}$-variety, then $w=w_{0,J}c$ for some Coxeter type element such that $\ell(w)=\ell(w_{0,J})+\ell(c)$ and ${\rm supp}(c)\cap {\rm supp}(w_{0,J})=\emptyset$.
\end{Proposition}

\begin{proof}
Since $X_{wB}$ is a horospherical $L_{J}$-variety, $X_{wB}$ is a spherical $L_{J}$-variety. Therefore by \cite[Theorem 6.5]{CanSaha2023} we have $w=w_{0,J}c$ for some Coxeter type element $c$ such that $\ell(w)=\ell(w_{0,J})+\ell(c)$.

Let $x \in X_{wB}$ be a point in the general position whose stabilizer contains the maximal unipotent subgroup $U_{L_{J}}$
of $B\cap L_{J}$. 
Without loss of generality, we may assume that $x$ is of the form 
$$
x=u\dot{v}B/B
$$ 
for some $v=v_1c_1$, where $v_1\le w_{0,J}$ and $c_1\le c$ with $\ell(v)=\ell(v_1)+\ell(c_1)$, 
and $u$ is in the unipotent subgroup $U_{v}:=\prod\limits_{\beta \in \mc{R}^+(v^{-1})}U_{\beta}$.

We now make a crucial observation regarding the factor $v_1$ of $v$.
If $v_1\neq id$, then $x$ cannot be fixed by all the elements of $U_{L_{J}}$. 
For example, let $\alpha\in \mc{S}\cap \mc{R}^+(v_{1}^{-1})$. 
Then we have 
$$
U_{\alpha}(1)x\neq x.
$$ 
To see this, notice that $U_{\alpha}(1)u\in U_{v}$ and that $u\in U_{v}$.
Then we have $u^{-1}U_{\alpha}(1)u\in U_{v}\setminus \{id\}$, showing that $U_\alpha(1)$ cannot fix $x$. 
Since $U_{L_J}$ is contained in the stabilizer of $x$, we conclude that $v_1 = id$. 

Now, since $v=c_1$, we know that $v\le c$. 
Since $\dim L_{J}\cdot x=\dim X_{wB}=\ell(w)$ and $U_{L_{J}}$ fixes $x$, by the dimensionality  arguments that we used in the last part of Proposition~\ref{P:horo1}, we see that $c_1= c$. 
Let $s_{i_1}s_{i_2}\cdots s_{i_{r}}$ be a reduced expression of $c$. 
As in the proof of Proposition~\ref{P:horo1}, for $1\le j\le r$, let us denote by $\beta_{i_j}$ the root 
$$
\beta_{i_j}:=s_{i_1}s_{i_{2}}\cdots s_{i_{j-1}}(\alpha_{i_{j}}).
$$
Then our point $x=u\dot{c} B/B$ is written in the form 
$$
x=U_{\beta_{i_1}}(a_1)\cdots U_{\beta_{i_r}}(a_r) \dot{c}B/B
$$ 
for some $(a_1,\ldots ,a_r)\in (k^{\times})^{r}$.
Let us consider the set 
$$
V:=\bigg\{U_{\beta_{i_1}}(a_1)U_{\beta_{i_2}}(a_2)\cdots U_{\beta_{i_r}}(a_r)\dot{c}B/B \mid (a_1,a_{2},\ldots ,a_r)\in (k^{\times})^{r} \bigg\}.
$$ 
Then $V$ is the open $T$-orbit of $x$ for the left action of $T$ on the toric Schubert variety $X_{cB}$. 
Since $U_{L_{J}}$ fixes the point $x\in V$ and $T$ normalizes $U_{L_{J}}$, it follows that $U_{L_{J}}$ fixes all points of $V$. 
But $V$ is a dense subset of $X_{cB}$.
It follows that $U_{L_{J}}$ acts trivially on $X_{cB}$.
Therefore, $U_{L_{J}}$ acts trivially on $X_{s_{i_{j}}B}$ for every $j\in \{1,\dots, r\}$.
This means that there are no simple roots $\alpha \in J$ such that $s_\alpha = s_{i_j}$ for some $j\in \{1,\dots,r\}$. 
In other words, we have ${\rm supp}(c)\cap {\rm supp}(w_{0,J})=\emptyset$.
Hence, the proof of our proposition is complete. 
\end{proof}

We are now ready to present a proof of the first main result of our article, Theorem~\ref{intro:T1}.
We recall its statement for convenience. 
\medskip

Let $J\subseteq \mc{I}_{w}$. The Schubert variety $X_{wB}$ is a horospherical $L_{J}$-variety if and only if $w=w_{0,J}c$ for some Coxeter type element such that $\ell(w)=\ell(w_{0,J})+\ell(c)$ and ${\rm supp}(c)\cap {\rm supp}(w_{0,J})=\emptyset$. 
\medskip

\begin{proof}[Proof of Theorem~\ref{intro:T1}]
The proof follows from Propositions~\ref{P:horo1} and~\ref{P:horo2}.
\end{proof}

We now proceed to state some corollaries of Theorem~\ref{intro:T1}.

\begin{Corollary}\label{C:productofintervals}
Let $J\subseteq \mc{I}_w$. Let $X_{wB}$ be a horospherical $L_{J}$-variety with $w= w_{0,J} c$, where $\ell(w) = \ell(w_{0,J})+\ell(c)$
and ${\rm supp}(c) \cap {\rm supp}(w_{0,J}) = \emptyset$. 
Then there is a poset isomorphism between the lower interval $[id, w] \subseteq W$ and the product of Bruhat-Chevalley orders,
$$
[id,w] \ \xrightarrow{\scaleobj{1.5}{\sim}} \ [id, w_{0,J}]\times [id, c].
$$
Furthermore, the interval $[id, w_{0,J}]$ is isomorphic to the Bruhat-Chevalley order on the parabolic subgroup $W_{J}$,
and the interval $[id, c]$ is isomorphic to the Boolean lattice on $\{1,\dots, r\}$, where $r=\ell(c)$.
\end{Corollary}

\begin{proof}
Let $z\in [id, w]$.
Since ${\rm supp}(w_{0,J}) \cap {\rm supp}(c) = \emptyset$, it readily follows from the subword property of Bruhat-Chevalley order that $z$ has a reduced expression of the form 
$$
z= \underbrace{(s_{m_1}\cdots s_{m_p})}_u\underbrace{(s_{n_1}\cdots s_{n_q})}_v,
$$
where $u\leqslant w_{0,J}$ and $v\leqslant c$.
It is easy to check that the map defined by $z\mapsto (u,v)$, $z\in [id, w]$ gives the desired poset isomorphism.

Our second assertion follows from the fact that $w_{0,J}$ is the maximal element of $W_{J}$, and 
the subword property of the Bruhat-Chevalley order on the interval $[id, c]$.
This finishes the proof of our corollary.  
\end{proof}

\begin{Corollary}
We assume that $G$ is of type $A$, $D$, or $E$. Let $J\subseteq \mc{I}_{w}$. 
Let $X_{wB}$ be a horospherical $L_{J}$-variety such that $w= w_{0,J} c$, where $\ell(w) = \ell(w_{0,J}) + \ell(c)$
and ${\rm supp}(c) \cap {\rm supp}(w_{0,J})= \emptyset$. 
Then the intersection of $X_{wB}$ with the $w_0$-translated Schubert variety $w_0 X_{w_0 w_{0,J} B}$ is a nonsingular toric variety. 
\end{Corollary}
Before proving our result, we point out that the intersection $X_{wB} \cap w_0 X_{w_0 w_{0,J} B}$ is the {\em Richardson variety}, denoted by $X_w^{w_{0,J}}$. 
\begin{proof}
In~\cite[Theorem 1.1]{CanSaha2023Toric}, it is shown that a Richardson variety $X_w^v$, where $\ell(w)-\ell(v) \leq \dim T$, 
is a nonsingular toric variety if and only if the interval $[v,w]$ is a Boolean lattice. 
In our situation, it follows from Corollary~\ref{C:productofintervals} that the interval $[w_{0,J}, w]$ is isomorphic to the lower interval $[id, c]$.
Since $c$ is a Coxeter type element, $[id, c]$ is a Boolean lattice. 
Hence, our claim follows. 
\end{proof}

It worths mentioning here that in~\cite{CanSaha2023Toric}, Can and Saha characterized the toric Richardson varieties.

\begin{Corollary}\label{nonsingularhoro}
Let $w\in W$ and  $J\subseteq S$.
If $c\in W$ is a Coxeter type element of the form $c=s_{i_{1}}s_{i_{i_{2}}}\cdots s_{i_{r}}$ such that 
\begin{enumerate}
\item $w=w_{0,J}c$,
\item $\ell(w)=\ell(w_{0,J})+\ell(c)$,
\item and $s_{i_{j}}s_{i_{k}}=s_{i_{k}}s_{i_{j}}$ for every $\{j,k \}\subset \{1,\dots, r\}$,
\end{enumerate}
then the Schubert variety $X_{wB}$ is a nonsingular horospherical $L_{J}$-variety.
\end{Corollary}
\begin{proof}
It follows from the fact that the simple reflections that appear in $c=s_{i_{1}}s_{i_{i_{2}}}\cdots s_{i_{r}}$ commute with each other, and since $\ell(w)=\ell(w_{0,J})+\ell(c)$ holds, the supports of $c$ and $w_{0,J}$ are disjoint. 
Hence, Theorem~\ref{intro:T1} implies that $X_{wB}$ is a horospherical $L_{J}$-variety. 
The nonsingularity of $X_{wB}$ can be seen in several different ways.
We apply~\cite[Theorem 1.2]{CanSaha2023}, which states that if $X_{wB}$ is a spherical $L_{J}$-variety,
then $X_{wB}$ is nonsingular if and only if $X_{c^{-1}P_{J}}$ is a nonsingular toric variety in $G/P_{J}$.
To show that $X_{c^{-1}P_{J}}$ is nonsingular, we first prove that there is an isomorphism,
$$
f: X_{c^{-1}P_{J}} \longrightarrow \prod_{j=1}^r X_{s_{i_j}P_{J}}.
$$
The map $f$ is defined as follows. 
Let $x\in X_{c^{-1}P_{J}}$. 
By the Bruhat-Chevalley decomposition, $x$ has a unique expression of the form $u \dot{c}'P_{J}/P_{J}$,
where $u\in U$, and $c'$ is a Coxeter type element such that $c' \leq c^{-1}$.
Note also that the commutation relations $s_{i_{j}}s_{i_{k}}=s_{i_{k}}s_{i_{j}}$, where $\{j,k \}\subset \{1,\dots, r\}$,
implies that $c^{-1} = c$. 
Then $f(x)$ is defined by 
$$
f(x) := (a_1P_{J},a_2P_{J},\dots, a_r P_{J}),
$$
where 
$$
a_j = 
\begin{cases}
u P_{J} & \text{ if $s_{i_j}\leq c'$},\\
1 P_{J} & \text{ if $s_{i_j}\nleq c'$}.
\end{cases}
$$
It is easy to check that $f$ is an isomorphism. 
Since $X_{s_{i_j}P_{J}} \cong \PP^1$ for each $j\in \{1,\dots, r\}$,
we see that $\prod_{j=1}^r X_{s_{i_j}P_{J}}$ is isomorphic to the product of projective spaces, hence nonsingular. 
This finishes the proof of our assertion. 
\end{proof}

\begin{Remark}
Let $w\in W$ and $J\subseteq J'\subseteq \mc{I}_{w}$. If $X_{wB}$ is a horospherical $L_{J}$-variety then $X_{wB}$ is also a horospherical $L_{J'}$-variety. However the following example demonstrates that the converse is not true.
\end{Remark}
\begin{Example} 
In this example, we will use Theorem~\ref{intro:T1} to show that full flag variety $SL(3,k)/B$ is horospherical for the action of $SL(3, k)$ but it is not horospherical variety for the action of $L_{\alpha_1}$ or $L_{\alpha_2}$, where $\{\alpha_1,\alpha_2\}$ is the set of simple roots associated with $B$ in $SL(3,k)$.  
We begin with $L_{\alpha_1}$. 

We notice, by using~\cite[Theorem 1.1]{CanSaha2023} that $SL(3,k)/B$ is a spherical $L_{\alpha_1}$-variety:
Let $J= \{\alpha_1\}$. 
Then the longest element of the Weyl group of $L_{\alpha_1}$ is given by $s_1$. 
Let $w$ denote $s_1s_2s_1$ so that we have $X_{wB}= SL(3,k)/B$.
It follows from the decomposition 
$$
w= \underbrace{s_1}_{w_{0,J}} \underbrace{s_2s_1}_{c}
$$ 
that 1) $X_{wB}$ is a spherical $L_{\alpha_1}$-variety as claimed, 2) the support of $w_{0,J}$ and $c$ is not disjoint. 
Therefore, by Theorem~\ref{intro:T1}, we see that $X_{wB}$ is not a horospherical $L_{\alpha_1}$-variety. 
The proof of the fact that $X_{wB}$ is not a horospherical $L_{\alpha_2}$-variety is similar. 
\end{Example}

\section{Existence of Horospherical Varieties with Prescribed Stabilizers}\label{S:Existence}

In this section, we address the following question:

\medskip

Given a set of simple roots $\mc{J}$, does there exist a horospherical variety $X_{wB}$ such that ${\rm stab}_G(X_{wB}) = P_{\mc{J}}$?

\medskip

When $\mc{J}$ is empty, the situation simplifies. 
We consider the point Schubert variety $X_{{id}B}$, where ${id}$ denotes the identity element. 
This variety has a stabilizer in $G$ that coincides with the Borel subgroup $B$ itself. 
In other words, it corresponds to the standard parabolic subgroup $P_\emptyset$. 
Consequently, if $\mc{J} = \emptyset$, then the answer to our question is affirmative. 
In fact, we will show that, for simple algebraic groups $G$, the answer is always affirmative.

\begin{Theorem}\label{T:prescribed}
Let $G$ be a simple algebraic group. 
Let $\mc{J}$ be a nonempty set of simple roots from $\mc{S}$. 
Then there is a nonsingular Schubert variety $X_{wB} \subseteq G/B$ such that 
\begin{itemize}
	\item [(i)] $P_{\mc{J}}={\rm stab}_{G}(X_{wB});$
	
	\item [(ii)] $X_{wB}$ is a $L_{\mc{J}}$-horospherical variety.
\end{itemize}
\end{Theorem}

\begin{proof}
If $\mc{J}=\mc{S}$, then we take $w=w_{0}$. 
For this choice, we have $P_{\mc{J}}=L_{\mc{J}}=G$ and $X_{wB}=G/B$.
Clearly, $G/B$ is a horospherical $L_{\mc{J}}$-variety. 
Hence, in this case, both of our assertions hold true. 

We proceed with the assumption that $\mc{J} \neq \mc{S}$. 
Since $G$ is simple there is $\alpha$ in $\mc{S}\setminus \mc{J}$ and a $\beta$ in $\mc{J}$ such that $\langle \alpha, \beta \rangle\neq 0$.
Here, $\langle \cdot , \cdot \rangle$ is the $W$-invariant inner product on the Euclidean space $X(T)\otimes_\Z \R$, where $X(T)$ is the character group of $T$.   
We now consider the Weyl group element defined by $w:=w_{0,\mc{J}} s_{\alpha}$. 
Then $\ell(w)=\ell(w_{0,\mc{J}})+1$.
Notice that the left descent set of $w$ is precisely the set of simple reflections, $\{s_\gamma \mid \gamma \in \mc{J} \}$.
Therefore, the stabilizer of $X_{wB}$ is the parabolic subgroup $P_{\mc{J}}$. 
Now since we know that $L_{\mc{J}}$ is the standard Levi subgroup of the stabilizer of $X_{wB}$ in $G$, Theorem~\ref{intro:T1} is applicable to our situation. It follows that $X_{wB}$ is a horospherical $L_{\mc{J}}$-variety. Finally, 
by using Corollary~\ref{nonsingularhoro}, we conclude that $X_{wB}$ is nonsingular. 
This finishes the proof of our theorem.
\end{proof}

\section{Generic Stabilizers are Strongly Solvable Subgroups}\label{S:Generic}

Recall that a subgroup $H \subseteq G$ is called strongly solvable if it is contained in a Borel subgroup $B$ of $G$. 
These subgroups give rise to interesting fiber bundles on the partial flag varieties, $G/H \to G/P$, where $P$ is a parabolic subgroup containing $B$. 
Luna~\cite{Luna1993} and Avdeev~\cite{Avdeev2011} elucidated the significance and importance of such homogeneous spaces for the embedding theory of spherical varieties. 
The purpose of this section is to explore strongly solvable subgroups in the context of Schubert varieties.
In particular, we provide a proof of our second main result.
We restate it here for the convenience of the reader. 

\medskip

Let $X_{wB}$ be a Schubert variety. Let $J\subseteq \mc{I}_w$. 
Then the stabilizer of any point $x$ in $L_{J}$ is a strongly solvable subgroup of $L_{J}$.

\begin{proof}[Proof of Theorem~\ref{intro:T2}]
Let $P$ denote, as before, the stabilizer of $X_{wB}$ in $G$. 
Let $x\in X_{wB}$. 
We set:
$$
H_x:={\rm stab}_{L_{J}}(x)\quad\text{and}\quad Q:={\rm stab}_G(x).
$$
Then $Q$ is a Borel subgroup of $G$.
The following equality of groups is evident:
\begin{align}\label{A:ourinclusion}
H_x = Q \cap L_{J}.
\end{align}
It is also clear that $H_{x}\subseteq \big((P\cap Q)\cdot {\mc {R}}_{u}(P)\big)\cap L_{J}$.
The ambient group $\big((P\cap Q)\cdot {\mc{R}}_{u}(P)\big)\cap L_{J}$ is a Borel subgroup of $L_{J}$. 
This follows from~\cite[Proposition 21.13 (i)]{Borel}.
This shows that $H_x$ is contained a Borel subgroup of $L_{J}$.
Hence, the proof of our assertion is finished. 
\end{proof}

\begin{Remark}\label{R:notclear}
	We maintain the notation of the previous proposition. 
	It is not clear whether $H_x$ is a connected subgroup of the Levi subgroup $L_{J}$ or not.
	It turns out that this question not only depends on the point in question, but also on the type of the ambient semisimple group $G$. 
	We will discuss this question in greater detail in Section~\ref{S:Connectedness}.
\end{Remark}

We now record a consequence of Theorem~\ref{intro:T2}.

\begin{Proposition}\label{P:ssss} Let $w\in W$ and $J\subseteq \mc{I}_w$.
If $X_{wB}$ is a spherical $L_{J}$-variety, then the stabilizer of a point in general position is a strongly solvable spherical subgroup of $L_{J}$. 
\end{Proposition}

\begin{proof}
Let $B_w$ denote the intersection $B\cap L_{J}$. 
Since $L_{J}$ is the standard Levi factor of ${\rm stab}_G(X_{wB})$, $B_w$ is a Borel subgroup of $L_{J}$. 
Let $x$ be a point from the open $B_w$-orbit in $X_{wB}$. 
Let $H_x$ denote the stabilizer of $x$ in $L_{J}$. 
Then the $L_{J}$-orbit of $x$, which is isomorphic to $L_{J} / H_x$, is a spherical homogeneous space. 
In particular, the Borel subgroup $B_w$ has only finitely many orbits in $L_{J}/H_x$, showing that $H_x$ is a spherical subgroup. 
The rest of our assertion follows from Theorem~\ref{intro:T2}.
\end{proof}

\begin{Remark}
In their work~\cite{GandiniPezzini}, Gandini and Pezzini introduced a combinatorial model for the set of Borel orbits in a strongly solvable spherical homogeneous space $G/H$, employing weight polytopes. It would be interesting to apply their work in the context of generic stabilizers of Schubert varieties. 
\end{Remark}

\section{Closed Orbits}\label{S:Closed}

A nonsingular complete $G$-variety $X$ is called a {\em wonderful variety} if it satisfies the following properties:
\begin{enumerate}
    \item $X$ contains a dense $G$-orbit $X_0$ whose complement is the union of a finite number of nonsingular prime divisors $D_i$ $(i \in I)$, each of which is $G$-invariant.
    \item The $D_i$ are nonsingular and have normal crossings, meaning that every non-empty intersection $D_{i_1} \cap \cdots \cap D_{i_k}$ is nonsingular and of codimension $k$.
    \item The $G$-orbit closures in $X$ are precisely the non-empty intersections $D_{i_1} \cap \cdots \cap D_{i_k}$.
\end{enumerate}
In this case, a subgroup $H\subset G$ such that $G/H \cong X_0$ is called a {\em wonderful subgroup},
and $X$ is said to be the {\em wonderful embedding of $G/H$}.
By an important result of Luna from~\cite{Luna1996}, we know that every wonderful variety is a spherical variety.
\medskip

The {\em rank} of a wonderful variety can be introduced as follows.  
Let $X$ be a wonderful variety as we defined above. 
Then $X$ has a unique closed $G$-orbit, which is given by $\bigcap_{i\in I} D_i$. 
This closed orbit is necessarily a partial flag variety since $X$ is complete.
The number of boundary divisors, $r:= | I |$, is equal to the codimension of the closed orbit.
This common number is called the {\em rank} of $X$.
 
\medskip

Let $G/H$ be a spherical homogeneous space. 
Let $\mc{D}$ denote the set of colors of $G/H$. 
Hence $\mc{D}$ is the $B$-stable divisors of $G/H$. 
The normalizer of $H$ in $G$ has a natural action on $\mc{D}$. 
The kernel of this action of $N_G(H)$ on $\mc{D}$ is called the {\em spherical closure} of $H$. 
We denote the spherical closure of $H$ by $\overline{H}$.
The spherical subgroup $H$ is called {\em spherically closed} if $\overline{H}=H$ holds. 
Recall that a spherical subgroup $H$ is called wonderful if $G/H$ admits a wonderful embedding. 
Equivalently, by a result of Knop, $H$ is wonderful if and only if the lattice of weights of $B$-semiinvariant rational functions on $G/H$ is generated by the spherical roots of $G/H$. 
An important result of Avdeev, in the context of strongly solvable spherical subgroups states that  
$H$ is a wonderful subgroup if and only if it is spherically closed.
\medskip

Let $w\in W$ and $J\subseteq \mc{I}_w$. Let $X_{wB}$ be a horospherical $L_{J}$-variety.
Let $H_x$ denote the generic stabilizer of a horospherical $L_{J}$-variety $X_{wB}$. 
Then $H_x$ is a horospherical subgroup of $L_{J}$.
By~\cite[Satz 2.1]{Knop1990}, we know that the normalizer $N_{L_{J}}(H_x)$ is a parabolic subgroup containing $H_x$. 
Since $H_x$ is also a strongly solvable subgroup of $L_{J}$, we see that $N_{L_{J}}(H_x)$ is the Borel subgroup 
of $L_{J}$.
Hence, we conclude that the quotient group $N_{L_{J}}(H_x)/ H_x$ is a torus.

\begin{Remark}
In general, a spherical subgroup $H\subseteq G$ is called {\em sober} if the group $N_G(H)/H$ is finite. 
Notice that, if a sober subgroup $H$ contains a maximal unipotent subgroup, then it must contain its normalizer as well. 
This means that a sober subgroup is horospherical if and only if it is a parabolic subgroup. 
By a result of Brion and Pauer~\cite{BrionPauer}, wonderful subgroups are known to be sober.
Hence, the only wonderful horospherical subgroups are the parabolic subgroups. 
In this case, the rank of the wonderful embedding is necessarily 0.
\end{Remark}

Returning to the generic stabilizers of horospherical Schubert varieties, we know that if a generic stabilizer subgroup $H_x \subset {L_{J}}$ is a non-parabolic horospherical subgroup, then it is not sober.
In particular, by our previous remark, we know that $H_x$ is not a wonderful subgroup. 
Thus, within the scope of our discussion about generic stabilizers, the following key problems should be addressed:
\begin{enumerate}
\item Identify the elements $w$ of $W$ for which the generic stabilizer of $X_{wB}$ in $L_{\mathcal{I}_w}$ is a wonderful subgroup.
\item For $w \in W$ as identified in part (1), describe the structure of the corresponding wonderful variety $X_{wB}$.
\end{enumerate}
We will provide a complete solution to both of these questions.

\begin{Definition}
A $G$-variety $X$ is called a {\it simple $G$-variety} if it has a unique closed $G$-orbit.
\end{Definition}

\begin{Theorem}\label{T:itissimple}
Let $w\in W$ and $J\subseteq \mc{I}_{w}$. Let $L_J$ denote the corresponding standard Levi subgroup.
Then 
$X_{wB}$ is a simple $L_{J}$-variety if and only if $w=w_{0,J}$. In particular $J=\mc{I}_w$.
\end{Theorem}
\begin{proof}
($\Rightarrow$) 
Let $Y$ denote the unique closed $L_J$-orbit in $X_{wB}$. 
The ``origin'' of the Schubert variety $X_{wB}$ is given by the point $B/B$. 
Let $B_J$ denote $L_J \cap B$. 
The $L_J$-orbit of $B/B$ is isomorphic to the flag variety $L_J/B_J$. 
Indeed, we have 
$$
L_J B/B \cong L_J/ L_J \cap B = L_J/B_J.
$$
Since $L_J/B_J$ is a complete variety, it is closed in $X_{wB}$.
In other words, $L_J\cdot B/B$ is the unique closed $L_J$-orbit, $Y$.  
At the same time, since $L_J$ acts transitively on $L_J/B_J$, and since $w_{0,J}$ is the maximal element of the Weyl group of $L_J$, we have $L_J\dot{w}_{0,J}B_J/B_J=L_J/B_J$. 
This means that $L_J \cdot B/B = L_J \cdot \dot{w}_{0,J}B/B$, implying the relation $w_{0,J}\leq w$ in $W$. 
Hence, we see that 
$$
w=w_{0,J}v,
$$ 
where $v$ is an element of $W$ such that $\ell(w)=\ell(w_{0,J})+\ell(v)$.
Note that $v$ cannot be a nontrivial element of the Weyl group of $L_J$. 
Otherwise the lengths would not add up to $\ell(w)$. 
We claim that $v= id$. 
To show this, let $H:={\rm stab}_{L_{J}}(\dot{v}B/B)$. 
Then, clearly, we have the inclusion $T\subset H$. Moreover, since $\ell(w)=\ell(w_{0,J})+\ell(v)$, it follows that $U_{\alpha}\dot{v}B/B=\dot{v}B/B$ for all $\alpha\in J$. 
Note that $U_{-\beta}\dot{u}B/B\neq\dot{u}B/B$ for all roots $\beta$ associated to $B_{J}$.
Hence, $H$ is the Borel subgroup $H=B_{J} \subset L_J$. 
In other words, we have an isomorphism, $L_J \cdot \dot{v}B/B \cong L_J/B_J$, implying the existence of the relation $v \leq w_{0,J}$. 
Hence, it follows from our previous discussion that $v= id$. 
In other words, we have $w = w_{0,J}$. 
This finishes the proof of the ``only if'' part of our assertion.
	
($\Leftarrow$) Assume that $w= w_{0,J}$. We already observed in the proof of the previous implication that $X_{w_{0,J}B} \cong L_J/B_J$, where $B_J = B\cap L_J$. 
Since $L_J/B_J$ is a simple $L_J$-variety, the proof of our lemma is finished. 
\end{proof}

Notice that Theorem~\ref{T:itissimple} answers the first question we posed above. 
In fact, it answers the second question as well.  

\begin{Corollary}\label{C:itissimple}
Let $w\in W$ and $J\subseteq \mc{I}_w$. 
The Schubert variety $X_{wB}$ is a wonderful $L_J$-variety if and only if $w = w_{0,J}$. In this case, $X_{wB}$ is isomorphic to the full flag variety of $L_{\mc{I}_w}$, which is a wonderful variety of rank 0. 
\end{Corollary}

\begin{proof}
Let us assume that $X_{wB}$ is a wonderful $L_J$-variety. Then $X_{wB}$ is a simple $L_J$-variety. 
Hence, by Theorem~\ref{T:itissimple}, we have $w= w_{0,J}$. 
Conversely, if $w= w_{0,J}$, then $X_{wB}= L_J/B_J$. 
But every flag variety is a wonderful variety. 
This finishes the proof of our assertion.
\end{proof}

\medskip

For $J\subseteq S$, let $L_J$ be the standard Levi subgroup of the standard parabolic subgroup $P_{J}$ of $G$ associated to $J$. 
Every closed $L_J$-orbit in $G/B$ contains at least one $T$-fixed point.
Since the number of closed $T$-orbits in $G/B$ are indexed by the elements of the Bruhat-Chevalley poset on $W$, 
we see that the number of closed $L_J$-orbits in $G/B$ will always be less than or equal to the cardinality $|W|$. 
In other words, there are only finitely many closed $L_J$-orbits in $G/B$. 
Before discussing a more precise generalization of this observation we look at two extreme examples.

\begin{Example}
Let $L_J$ be the smallest standard Levi subgroup of $G$, that is, the maximal torus $T$. 
Evidently, the closed $T$-orbits in $G/B$ are its fixed points, 
$$
\dot{w}B/B\qquad (w\in W).
$$
Although $T$ has only finitely many closed orbits, it has infinitely many orbits. 
\end{Example}

\begin{Example}
Let $p$ and $q$ be two positive integers.
Let $n$ denote $p+q$. 
Let $L_{p,q}$ denote the Levi subgroup $L_{p,q}:=GL(p,\C)\times GL(q,\C)$ of $GL(n,\C)$.
Let $B$ denote the Borel subgroup of all upper triangular matrices in $GL(n,\C)$. 
The left $L_{p,q}$-orbits in $GL(n,\C)/B$ are in one-to-one correspondence with the right $B$-orbits in $L_{p,q}\backslash GL(n,\C)$.
It is well-known that $L_{p,q}$ is a spherical subgroup.
Therefore, the number of $B$-orbits in $L_{p,q}\backslash GL(n,\C)$ is finite.
Equivalently, the number of $L_{p,q}$-orbits in $GL(n,\C)/B$ is finite. 
Let $S_n$ denote the symmetric group.
Then the Weyl group of $L_{p,q}$ is the parabolic subgroup $S_p\times S_q$ of $S_n$.
As a corollary of our next result, we will see that the number closed $L_{p,q}$-orbits in $GL(n,\C)/B$ is given by 
the number of minimal length right coset representatives of $S_p\times S_q$ in $S_n$. 
\end{Example}

Let $W_{J}$ denote the Weyl group of $L_{J}$. Let $^{J}W$ denote the set of all minimal right coset representatives of $W_{J}\backslash W$ in $W$.
\begin{Proposition}\label{P:noofclosed}
Let $w\in W$ and $J\subseteq \mc{I}_w$. If $v$ is the element of $W$ such that $w= w_{0,J}v$ and $\ell(w) = \ell(w_{0,J}) + \ell(v)$, 
then the number of closed $L_{J}$-orbits in $X_{wB}$ is given by the cardinality, 
\begin{align*}
|\big\{u\le v : \ell(w_{0, J}u)=\ell(w_{0,J})+\ell(u))\big\}|.
\end{align*}
\end{Proposition}
\begin{proof}
The proof is similar in part to the ``if'' direction of the proof of Theorem~\ref{T:itissimple}. 
We will show that the closed $L_J$-orbits in $X_{wB}$ are parametrized by the elements $u\in W$ satisfying 
the properties:
1) $u\leq v$, and 2) $\ell(w_{0,J}u)=\ell(w_{0,J})+\ell(u)$.
Evidently, such an element $u\in W$ cannot be a nontrivial element of $W_{J}$. 
Otherwise, we find a contradiction to the equality $\ell(w_{0,J}u)=\ell(w_{0,J})+\ell(u)$. 
Let $H:={\rm stab}_{L_{J}}(\dot{u}B/B)$. 
Clearly, we have the inclusion $T\subset H$. 
Also, since $\ell(w_{0,J}u)=\ell(w_{0,J})+\ell(u)$, we know that $U_{\alpha}\dot{u}B/B=\dot{u}B/B$ for all $\alpha\in J$.
Note that $U_{-\beta}\dot{u}B/B\neq\dot{u}B/B$ for all roots $\beta$ associated to $B_{J}=B\cap L_{J}$.
Hence, $H$ is the Borel subgroup $H=B_{J} \subset L_J$. 
In other words, we have an isomorphism, $L_J \cdot \dot{u}B/B \cong L_J/B_J$.
Since $L_J/B_J$ is complete, the orbit $L_J \cdot \dot{u}B/B$ is a closed orbit.
Conversely, if $Y$ is a closed $L_J$-orbit in $X_{wB}$, then we can assume that $Y$ is the $L_J$-orbit of a $T$-fixed point since every closed $L_J$-orbit contains a $T$-fixed point. 
The rest of the proof follows from the argument we used for the first part. 
\end{proof}

\begin{Corollary}
The number of closed $L_J$-orbits in $G/B$ is given by the cardinality of $^{J}W$. 
\end{Corollary}

\begin{proof}
The maximal element of $W$, that is, $w_0$ has a decomposition of the form $w_0= w_{0,J}{^{J}w_0}$, where $^{J}w_0$ is the maximal element of $^{J}W$. 
Proposition~\ref{P:noofclosed} implies that the number of closed $L_J$-orbits in $X_{w_0B}$ is given by the cardinality of the set 
$\big\{u\le {^{J}w_0}: \ell(w_{0, J}u)=\ell(w_{0,J})+\ell(u))\big\}$, which is exactly the set $^{J}W$. 
Since the Schubert variety $X_{w_0B}$ is $G/B$, our assertion follows.
\end{proof}

\section{The Connectedness Question in Detail}\label{S:Connectedness}

As we promised in Remark~\ref{R:notclear}, the purpose of this section is to discuss the connectedness of the stabilizer group of a point in general position in a Schubert variety. 
We begin with a preliminary observation on the joint kernels of the roots that arise from reduced expressions.

\begin{Lemma}\label{lemma A}
	Let $w=s_{i_1}s_{i_2}\cdots s_{i_r}$ be a reduced expression for an element $w\in W$. 
	For $j\in \{1,\dots, r\}$, let $\beta_j$ denote the root defined by $\beta_j :=s_{i_1}s_{i_2}\cdots s_{i_{j-1}}(\alpha_{i_{j}})$. 
	Then we have 
	\[
	\bigcap\limits_{j=1}^{r}\ker \beta_{j}= \bigcap\limits_{j=1}^{r}\ker \alpha_{i_{j}}=\bigcap\limits_{\alpha\in {\rm supp}(w)}\ker \alpha.
	\]
\end{Lemma}
\begin{proof}
	It is evident that the inclusion $\bigcap\limits_{j=1}^{r}\ker \alpha_{i_{j}}\subseteq \bigcap\limits_{j=1}^{r}\ker \beta_{j}$ holds. 
	It is also evident that if $t\in \bigcap\limits_{j=1}^{r}\ker \beta_{j}$, then we have $\beta_{j}(t)=1$ for all $j\in \{1,\dots,r\}$. 
	In this case, the equality $\beta_{1}=\alpha_{i_{1}}$ implies that $\alpha_{i_{1}}(t)=1$. 
	Next, we check the character value $\alpha_{i_2}(t)$ under the same assumption.
	Since $\beta_{2}=\alpha_{i_{2}}-\langle\alpha_{i_{2}}, \alpha_{i_{1}} \rangle\alpha_{i_{1}}$, $\alpha_{i_{1}}(t)=1$, and $\beta_{2}(t)=1$, we see that 
	$$
	\alpha_{i_{2}}(t) = 1.
	$$
	
	Next, we focus on $\beta_3$. 
	Notice that $\beta_{3}=\alpha_{i_{3}}+\sum a_{1}\alpha_{i_{1}}+a_{2} \alpha_{i_{2}}$ for some non-negative integers $a_{1}$ and $a_{2}$.
	Since $\alpha_{i_{1}}(t)=1,$ $\alpha_{i_{2}}(t)=1$, and $\beta_{2}(t)=1$, by using an argument that is similar to the one that we used in the previous case, we see that $\alpha_{i_{3}}(t)=1$.
	Proceeding along similar lines, we conclude that $\alpha_{i_{j}}(t)=1$ for every $j\in \{1,\dots, r\}$.
	In other words, we have $t\in \bigcap\limits_{j=1}^{r}\ker \alpha_{i_{j}}$. 
	Hence. the equality 
	\[
	\bigcap\limits_{j=1}^{r}\ker \beta_{j}=\bigcap\limits_{j=1}^{r}\ker \alpha_{i_{j}}
	\]
	follows. 
\end{proof}

In our next lemma, we assume that $G$ is centerless.
We will explain by examples afterwards why this specialization is needed.

\begin{Lemma}\label{lemma B}
	Let $G$ be a complex semi-simple group of adjoint type. 
	Let $\{\alpha_{i_{1}},\ldots, \alpha_{i_{r}}\}$ be a subset of the set of simple roots.  
	Then the intersection 
	\[
	\bigcap_{j=1}^{r}\ker \alpha_{i_{j}}
	\] 
	is a torus in $G$. 
	in particular, it is connected as an algebraic group.
\end{Lemma}

\begin{proof}
	First note that if $\{\alpha_{i_{1}},\ldots, \alpha_{i_{r}}\}=\mc{S}$, then $\bigcap_{j=1}^{r}\ker \alpha_{i_{j}}=Z(G)$. 
	Since $G$ is of adjoint type, the center of $G$, denoted $Z(G)$, is the trivial group. 
	Hence, the intersection $\bigcap_{j=1}^{r}\ker \alpha_{i_{j}}$ is connected.
	
	Next, we assume that $\{\alpha_{i_{1}},\ldots, \alpha_{i_{r}}\}$ is a proper subset of $\mc{S}$. 
	We consider the following algebraic group homomorphism:
	\begin{align*}
		\pi: T &\longrightarrow \mathbb{G}_{m}^{r} \\
		t & \longmapsto (\alpha_{i_{1}}(t),\ldots, \alpha_{i_{r}}(t)).
	\end{align*}
	Evidently, $\pi$ is surjective, and the coordinate functions on $\mathbb{G}_m^r$ are given by the characters $\alpha_{i_1},\dots, \alpha_{i_r}$. 
	In particular, the character group of $\mathbb{G}_m^r$ can be identified with the subgroup of the character group of $T$. 
	Note that the kernel of $\pi$ is given by $\ker \pi= \bigcap_{j=1}^{r}\ker \alpha_{i_{j}}$.
	These groups fit into a short exact sequence of commutative algebraic group homomorphisms:
	$$
	1 \to \ker \pi \to T \to (\mathbb{G}_{m})^{r} \to 1.
	$$
	We pass to the short exact sequence of character groups:
	$$
	0\to X(\mathbb{G}_m^r)\to X(T)\to X(\ker \pi)\to 0.
	$$
	Note also that, since $G$ is of adjoint type, we know that $X(T)$ is spanned by the set of all simple roots $\alpha_1,\dots, \alpha_n$. 
	Hence, since $X(\mathbb{G}_{m}^{r})$ is freely generated by the characters $\alpha_{i_1},\dots, \alpha_{i_r}$,
	we see that $X(\ker \pi)$ is freely generated by the simple roots $\mc{S} \setminus \{\alpha_{i_1},\ldots, \alpha_{i_{r}}\}$. 
	Therefore, by duality, $\ker \pi$ is a subtorus of $T$.
\end{proof}

We will present examples to show that the assumption of adjointness on the semisimple group $G$ is indeed important for the conclusion of our lemmas. 

\begin{Example}
	
	We consider $SL(2,\C)$ as our semisimple algebraic group, $G$. 
	Let $B$ denote the Borel subgroup consisting of upper triangular matrices. 
	Let $T$ denote the diagonal torus in $B$. 
	The unique simple root determined by this data is denoted by $\alpha_1$. 
	Then the kernel of $\alpha_1$ is given by 
	$$
	\ker \alpha_1=\{{\rm diag}(t, t): t^2=1\}.
	$$
	Evidently, this is abelian group is not connected. 
\end{Example}

Next, we give an example to explain what happens for the reduced expressions as in Lemma~\ref{lemma A}.

\begin{Example}
	Let $G:=SL(3,\C)$.
	Let $B$ denote the Borel subgroup consisting of upper triangular matrices. 
	Let $T$ denote the diagonal torus in $B$. 
	The simple roots determined by this data are given by $\alpha_1$ and $\alpha_2$. 
	In particular, the product $s_{1}s_{2}$ is a reduced expression. 
	Let us compute the roots $\beta_1$ and $\beta_2$ as in Lemma~\ref{lemma A}:
	\begin{align*}
		\beta_{1}=\alpha_{1}\qquad \text{ and }\qquad \beta_{2}=s_1(\alpha_2) = \alpha_{1}+\alpha_{2}.
	\end{align*}
	Then the intersection of the kernels of $\beta_1$ and $\beta_2$ is given by the following diagonalizable group: 
	$$
	\ker \beta_1\cap \ker \beta_{2}=\{{\rm diag}(t, t, t): t^{3}= 1\}.
	$$
	Evidently, this is not a connected group.
\end{Example}

Finally, we are ready to state and prove a connectedness result that we promised earlier.

\begin{Proposition}
	Let $G$ be a complex semisimple algebraic group of adjoint type. 
	Let $w\in W$ and $J\subseteq \mc{I}_{w}$.
	If $X_{wB}$ is an $L_{J}$-horospherical variety, then the stabilizer of a point in a general position in $X_{wB}$ is a connected strongly solvable group. 
\end{Proposition}

\begin{proof}
	
	The strong solvability part of the stabilizer subgroup have already been proven. 
	We proceed to prove the connectedness. 
	The initial part of the proof is similar to the proof of Proposition~\ref{P:horo2}.
	
	Since $X_{wB}$ is $L_{J}$-horospherical variety, by our earlier theorem we have $w=w_{0,J}c$ such that $\ell(w)=\ell(w_{0,J})+\ell(c)$ and
	${\rm supp}(c)\cap {\rm supp}(w_{0,J})=\emptyset$. 
	We fix a reduced expression $c=s_{i_{1}}s_{i_{i_{2}}}\cdots s_{i_{r}}$. 
	For $j\in \{1,\dots, r\}$, let $\beta_j$ denote the root $\beta_{j} :=s_{i_1}s_{i_2}\cdots s_{i_{j-1}}(\alpha_{i_{j}})$. 
	Then we have 
	\[
	U\dot{c}B/B=U_{\beta_{i_{1}}}U_{\beta_{i_{2}}}\cdots U_{\beta_{i_{r}}}\dot{c}B/B.
	\]
	Let $\xi$ denote the point $\xi=u_{\beta_{i_{1}}}(1)\cdots u_{\beta_{i_{r}}}(1)\dot{c}B/B$.  
	Then the dimension of the orbit $L_{J}\cdot \xi$ is equal to $\dim X_{wB}$.
	Let $H={\rm stab}_{L_{J}}(\xi)$. 
	Let $T_c$ denote the intersection defined by $T_{c}:=\bigcap\limits_{j=1}^{r} \ker \beta_{i_{j}}$.
	Then we have $T_c \subset H$. 
	Since $X_{wB}$ is $L_{J}$-horospherical, we have $\big \langle U_{\alpha}: \alpha \in J\big\rangle\subset H$.
	Moreover since $\ell(w)=\ell(w_{0,J})+\ell(c)$ it follows $B_{J}\dot{\tau}B\dot{c}B\subset B\dot{\tau}\dot{c}B$ for all $\tau \neq id$ in $W_{J}$. Therefore it follows that $H\cap B_{J}\dot{\tau}B_{J}=\emptyset$ for all $\tau \neq id$ in $W_{J}$. 
	Consequently $H$ is generated by $T_c$ and the root subgroups indexed by the elements of $J$: 
	\[
	H=\big\langle T_c~, U_{\alpha} \mid \alpha \in J \big \rangle.
	\]	
	It follows from Lemma \ref{lemma A} that $T_c$ is a torus. 
	In fact, $T_c$ is a maximal torus of $H$.
	Hence, $H$ is connected strongly solvable subgroup of $L_{J}$. 
	This finishes the proof of our assertion.
\end{proof}

\section{Doubly-spherical Varieties}\label{S:Doubly}

A spherical $G$-variety $X$ is said to be {\em nearly toric} if the minimum codimension of a $T$-orbit in $X$ is 1. 
These varieties were introduced in~\cite{CanDiaz2024}, where their examples among Schubert varieties were determined. 
It was shown in~\cite{CanDiaz2024, CanSaha2023} that if a singular Schubert variety $X_{wB} \subset GL(n,\C)/B$ such that %$c_T(X_{wB}) = 1$
the minimum codimension of a $T$-orbit in $X$ is 1, is a nearly toric Schubert variety. 
In this section, we present, among other things, a classification of nearly toric Schubert varieties across all types.

\begin{Lemma}\label{L:nearlytoric}
Let $w\in W$ and $J\subseteq \mc{I}_{w}$. 
Let $X_{wB} \subseteq G/B$ be a nearly toric $L_J$-variety such that $\dim B_J = \dim X_{wB}$. 
Then $w$ is a product of the form $s_\alpha c$, where $s_\alpha$ is a simple reflection and $c$ is a Coxeter element such that $\ell(w) = \ell(c)+1$. 
In particular, we have $J = \{ \alpha \}$. 
\end{Lemma}

\begin{proof}
Then we have 
$$
\dim X_{wB} = \dim T + \dim U_J,
$$
where $U_J$ is the maximal unipotent subgroup of $B_J$. 
We assume that $X_{wB}$ is nearly toric. 

There is a point $x\in X_{wB}$ such that 
$$
\dim T\cdot x = \dim X_{wB}-1.
$$
It follows from the equality 
$$
 \dim T + \dim U_J = \dim T \cdot x +1
$$
that $|J| \leq 1$. 
We claim that $|J| = 0$ is not possible. 
Indeed, if we assume otherwise that $|J|=0$, then $w_{0,J} =id$, implying that $w$ is a Coxeter element.
Then $X_{wB}$ is a toric variety by the main result of~\cite{Karuppuchamy}.
Since we assumed that %$c_T(X_{wB})=1$,
the minimum codimension of a $T$-orbit in $X_{wB}$ is 1, this is absurd. 
Hence, we know that $J = \{\alpha \}$ for some simple reflection $\alpha \in \mc{I}_{w}$. 
This proves our second assertion. 
In light of~\cite[Theorem 6.1]{CanSaha2023}, our first assertion follows from the second one.
\end{proof}

We are now ready to prove the third main result of our article. 
Let us recall its statement for convenience. 
\medskip

Let $w\in W$. 
Then $X_{wB}$ is a nearly toric $L_{\mc{I}_w}$-variety if and only if $w$ is of the form $w= s_\alpha c$, where $s_\alpha \in {\rm supp}(c)$ and $\ell(w) = \ell(c)+1$.

\begin{proof}[Proof of Theorem~\ref{intro:T3}]

($\Rightarrow$) 
Let $X_{wB}$ be a nearly toric Schubert variety.
Then $X_{wB}$ is a spherical $L_J$-variety, where $L_J$ is a Levi subgroup for some $J\subseteq \mc{I}_w$.
There is no harm in assuming that $\dim X_{wB} = \dim B_{L_J}$, where $B_{L_J}$ is the Borel subgroup of $L_J$. 
Then our assertion readily follows from Lemma~\ref{L:nearlytoric}.

($\Leftarrow$) 
If $w$ is of the form $w= s_\alpha c$, where $s_\alpha \in {\rm supp}(c)$ and $\ell(w) = \ell(c)+1$, then~\cite[Theorem 6.2]{CanSaha2023} implies that $X_{wB}$ is a spherical $L_J$-variety, where $J= \{\alpha\}$. 
Since $s_\alpha \in {\rm supp}(c)$, we know that $w$ is not a Coxeter element, hence, $X_{wB}$ is not a toric variety.
At the same time, the dimension of the Borel subgroup of $L_J$ is given by 
$$
\dim B_{L_{J}} = \dim T + \dim U_\alpha  = \dim T +1.
$$
Hence, $T$ has a one codimensional orbit in $X_{wB}$. 
It follows that the minimum codimension of a $T$-orbit in $X_{wB}$ is 1.
Hence, $X_{wB}$ is a nearly toric variety. 
This finishes the proof of our theorem.
\end{proof}

The proof of our previous theorem gives more. 
It shows that there is possibly a smaller Levi subgroup of the stabilizer group in $G$ of a nearly toric Schubert variety $X_{wB}$. 
Indeed, we have the following direct consequence of the (proof of the) Theorem~\ref{intro:T3}.

\begin{Corollary}\label{C:alphaissufficient}
Let $X_{wB}$ be a nearly toric Schubert variety.
If $w$ is given by $w= s_{\alpha} c$, where $\alpha$ is a simple root and $\ell(w) = \ell(c)+1$, then $X_{wB}$ is a nearly toric $L$-variety,
where $L$ is the Levi subgroup of $G$ generated by $T$ and the root subgroups $U_{\pm \alpha}$.   
\end{Corollary}

\begin{proof}
The proof is identical with the `only if' part of the proof of Theorem~\ref{intro:T3}.
\end{proof}

We now turn our attention to investigating a new family of spherical varieties. 
Let us first recall the definition of this class of varieties.
Let $X$ be a spherical $G$-variety. 
If every $B$-orbit closure in $X$ is a spherical $L$-variety, where $L$ is a Levi subgroup of $G$, then we call $G$ a {\em doubly-spherical $G$-variety}.
Our goal in this section is to show that nearly toric Schubert varieties are doubly-spherical $L$-varieties, where $L\subseteq G$ is a Levi subgroup.

We are now ready to prove our last stated result from the introduction. 
We recall its statement for convenience. 
\medskip

Let $X_{wB}$ be a nearly toric Schubert variety. 
Then $X_{wB}$ is a doubly-spherical $L_{\mc{I}_w}$-variety.

\begin{proof}[Proof of Theorem~\ref{intro:T4}]
Since $X_{wB}$ is a nearly toric Schubert variety, there exists $\alpha \in \mc{I}_w$ and a Coxeter type element $c\in W$ such that $w= s_\alpha c$. 
In particular, $X_{wB}$ is a spherical $L_{\mc{I}_w}$-variety.
We fix a reduced expression of the form 
\begin{align}\label{A:areducedexp}
w = s_{\alpha} \underbrace{s_{i_1}\cdots s_{i_r}}_c.
\end{align}
We have $s_{\alpha} \in \{ s_{i_1},\dots, s_{i_r}\}$.
Otherwise, $w$ would a Coxeter-type element, hence $X_{wB}$ would be a toric $T$-variety. 
We set 
\begin{align}\label{A:Lissufficient}
L:= \langle T , U_\alpha , U_{-\alpha} \rangle. 
\end{align}
Then, by Corollary~\ref{C:alphaissufficient}, $X_{wB}$ is a nearly toric $L$-variety. 
Note that the Borel subgroup $B_L:= B\cap L$ of $L$ is contained in the Borel subgroup $B_{L_{\mc{I}_w}}:=B\cap L_{\mc{I}_w}$ of $L_{\mc{I}_w}$. 
Let $Y$ be a $B_{L_{\mc{I}_w}}$-orbit closure in $X_{wB}$. 
Then $Y$ is $B_L$-stable as well. 
We proceed with the assumption that $Y$ is not a toric $T$-variety. 
We will show that $Y$ is a spherical $L$-variety. 
To this end, since $X_{wB}$ is a spherical $L$-variety, it suffices to show that $Y$ is $L$-stable. 
But this follows readily from the fact that $Y$ is stable under $B_L$, hence, ${\rm stab}_{L_{\mc{I}_w}}(Y)$ contains $U_{\pm \alpha}$. 
Hence, the proof of our assertion is finished.
\end{proof}

\begin{Remark}
We will mention two special examples.
The flag variety of $SL(3,\C)$ is not only a doubly-spherical variety but also a nearly toric variety.
The flag variety of $SL(4,\C)$ is a doubly-spherical but not nearly toric Schubert variety. 
In particular, this example shows that the family of doubly-spherical Schubert varieties is strictly larger than the family of nearly toric Schubert varieties.
\end{Remark}

\section*{Acknowledgements}

The first author gratefully acknowledges partial support from the Louisiana Board of Regents grant (LEQSF(2023-25)-RD-A-21) and additional partial support from the Fulbright Portugal, U.S. Scholar Program. The second author acknowledges the Infosys Foundation for the partial financial support. The last author expresses gratitude to the Indian Institute of Science (IISc) for fostering a supportive and productive research environment, and acknowledges the National Board for Higher Mathematics (NBHM) Post Doctoral Fellowship with Ref. Number 0203/21(5)/2022-R\&D-II/10342.

\bibliographystyle{plain}
\bibliography{references1}

\end{document}